\documentclass[a4paper,fleqn,10pt]{article}
\usepackage{times}
\usepackage{amsmath,amsthm,authblk}
\numberwithin{equation}{section}
\usepackage{amssymb,stmaryrd}

\newcommand{\C}{C}

\renewcommand{\div}{\operatorname{div}}
\newcommand{\dual}[2]{\langle #1,#2\rangle}

\newcommand{\ff}{\ve f}

\newcommand{\G}{\mathcal G}
\newcommand{\HH}{\ve H}
\newcommand{\J}{\mathcal J}

\newcommand{\nn}{\ve\nu}
\newcommand{\norm}[1]{\left\|#1\right\|}

\newcommand{\pbar}{\overline p}
\newcommand{\pOmega}{\partial\Omega}
\newcommand{\pphi}{\ve\varphi}

\newcommand{\pstav}[1]{\mathcal P(#1)}

\newcommand{\pstavFix}[1]{\mathcal P^{#1}}

\newcommand{\R}{\mathbb R}
\newcommand{\scal}[2]{(#1,#2)}

\newcommand{\tr}{\operatorname{tr}}
\renewcommand{\tt}{\ve\tau}
\newcommand{\Ttau}{{\sigma}^\tau}
\newcommand{\Uad}{\mathcal U_{ad}}
\newcommand{\ubar}{\overline\uu}
\newcommand{\un}{u_\nu}
\newcommand{\uu}{\ve u}
\newcommand{\V}{\ve V}
\newcommand{\vbar}{\overline\vv}
\newcommand{\Vdiv}{\V_{\div}}
\newcommand{\ve}[1]{\boldsymbol{#1}}
\newcommand{\vv}{\ve v}
\newcommand{\W}{\ve W}
\newcommand{\weakly}{\rightharpoonup}
\newcommand{\wOmega}{\widehat\Omega}

\newcommand{\ww}{\ve w}

\newtheorem{theorem}{Theorem}
\newtheorem{lemma}{Lemma}
\newtheorem{corollary}{Corollary}
\newtheorem{remark}{Remark}

\begin{document}



\title{Stokes problem with a solution dependent slip bound:\\ Stability of~solutions with respect to~domains}


\author[1]{Jaroslav Haslinger}
\affil[1]{\small Department of Numerical Mathematics, Faculty of Mathematics and Physics, Charles University in~Prague, Sokolovsk{\'a}~83, 186 75 Praha 8, Czech Republic}
\author[2]{Jan Stebel\footnote{Corresponding author,~e-mail:~\textsf{jan.stebel@tul.cz}}}
\affil[2]{Institute of Novel Technologies and Applied Informatics, Faculty of Mechatronics, Informatics and Interdisciplinary Studies, Technical University Liberec, Studentsk{\'a} 1402/2, 461 17 Liberec 1, Czech Republic}

\date{}

\maketitle

{\bf Abstract:}
We study the Stokes problem in a bounded planar domain $\Omega$ with a~friction type boundary condition that switches between a slip and no-slip stage.
Unlike our previous work \cite{HaSt_given_slip}, in the present paper the threshold value may depend on the velocity field.
Besides the usual velocity-pressure formulation, we introduce an alternative formulation with three Lagrange multipliers which allows a more flexible treatment of the impermeability condition as well as optimum design problems with cost functions depending on the shear and/or normal stress.
Our main goal is to determine under which conditions concerning smoothness of $\Omega$, solutions to the Stokes system depend continuously on variations of $\Omega$.
Having this result at our disposal, we easily prove the existence of a solution to optimal shape design problems for a large class of cost functionals.



\maketitle                   







\section{Introduction}


This paper analyses one property of the Stokes system defined in $\Omega\subset\R^2$ with a slip type boundary condition, namely the continuous dependence of its solutions on the shape of $\Omega$.
This property plays the crucial role in the existence analysis of optimal shape design problems. 
The no-slip boundary condition, i.e. the vanishing velocity on the boundary, is  widely used in practice.
It characterizes the adhesion of the fluid on the solid wall.
This condition is acceptable for small velocities and on a macroscopic level. 
On the other hand, there are many situations (flow of the fluid on hydrophobic surfaces, polymer melts flow, problems with multiple interfaces, micro/nanofluidics etc.) where the slip of the fluid occurs.
To get a more realistic model, the slip has to be taken into account.  
For the physical justification of different types of slip conditions we refer to \cite{RR99} and \cite{HL03}.
The mathematical analysis of the Stokes and Navier-Stokes system with the slip and  leak boundary conditions has been done in \cite{fujita1994} and extended to non-stationary problems in \cite{fujita2002}.
The regularity  of solutions to the Stokes system with slip and leak boundary conditions has been established in \cite{saito}.
In \cite{BM13} the stick--slip condition is considered as an implicit constitutive equation on the boundary and the existence of weak solutions to Bingham and Navier-Stokes fluids is proven. 

Shape optimization involving fluid models with slip boundary conditions as the state problem is of a great practical importance.
Slip boundary conditions affect the velocity profile and hence  the velocity gradient of the fluid in the vicinity of the wall.
The velocity gradient is an important factor in  the transformation of the mechanical energy to heat, the process representing the energy loss. 
Shape optimization of the interior of hydraulic elements may reduce the velocity gradient resulting in energy savings.
In \cite{HaSt_given_slip}   a class of shape optimization problems for the Stokes system with the threshold boundary conditions involving a priori given slip bound has been studied. 
The existence result for the continuous setting of the problem and convergence analysis for appropriate discretizations of the continuous model have been established. 

Nevertheless it is known from experiments that the slip bound may depend on the solution itself, e.g. on values of the tangential component of the velocity.
The aim of this paper is to extend the existing stability results to this type of the slip boundary condition.
Besides the standard velocity-pressure formulation used in \cite{HaSt_given_slip} we present a new weak formulation adding another two Lagrange multipliers: one releasing the impermeability condition and the other regularizing the non-smooth slip functional.
This new formulation turns out to be useful in numerical solution of this problem.
Moreover, it enables us to approximate directly the normal and shear stress and to use these quantities as arguments of  appropriate  objective functionals to control the stress distribution along the slip part of the boundary.

The paper is organized as follows: in Section \ref{sec:formulation_vp} we present the velocity-pressure formulation of the Stokes system with a solution dependent slip bound.
Using fixed point arguments we prove that such problem has at least one solution for any slip bound represented by a continuous, positive and bounded from above function $g$.
If in addition, $g$ is Lipschitz continuous with modulus of Lipschitz continuity sufficiently small, then the solution is unique.
Section \ref{sec:formulation_4field} deals with a four-field formulation of the problem whose solution is represented  by the velocity $u$, pressure $p$, normal, tangential shear stress $\sigma^\nu$, and $\sigma^\tau$, respectively.
In Section \ref{sec:shape_stability} we prove that the graph of the respective generally multi-valued solution mappings considered as a function of the shape of the slip  part of the boundary, is closed in an appropriate topology.
On the basis of these results the existence of solutions to a class of optimal shape design problems will be proven in Section \ref{sec:shape_optimization}.


\section{The velocity-pressure formulation of the problem}
\label{sec:formulation_vp}

Unlike \cite{HaSt_given_slip}, where the slip bound was given, the present paper deals with a more general case, namely the slip bound will be a function of the tangential velocity.

Let $\Omega\subset\R^2$ be a bounded domain with the Lipschitz boundary $\pOmega$.
The slip boundary conditions are prescribed on an open, non-empty part $S$ of the boundary and the no-slip condition on $\Gamma=\pOmega\setminus\overline S$, $\Gamma\neq\emptyset$:
\begin{subequations}
\label{eq:stokes}
\begin{align}
-\Delta\uu + \nabla p &= \ff &&\mbox{ in }\Omega,\\
\div\uu &= 0 &&\mbox{ in }\Omega,\\
\uu &= \ve 0 &&\mbox{ on }\Gamma,\\
\label{eq:bc_imperm}
\un &= 0 &&\mbox{ on }S,\\
|\Ttau| & \le g(|u_\tau|) &&\mbox{ on }S,\\
u_\tau\neq 0 & \Rightarrow |\Ttau|=g(|u_\tau|) ~\&~ \exists\lambda\ge 0: u_\tau=-\lambda\Ttau &&\mbox{ on }S.
\end{align}
\end{subequations}
Here $\uu=(u_1,u_2)$ is the velocity field, $p$ is the pressure and $\ff$ is the external force.
Further, $\nn$, $\tt$ denote the unit outward normal, and tangential vector to $\pOmega$, respectively.
If $\ve a\in\R^2$ is a vector then $a_\nu:=\ve a\cdot\nn$, $a_\tau:=\ve a\cdot\tt$ is its normal, and the tangential component on $\pOmega$, respectively.
Finally, $\Ttau:=\left(\frac{\partial\uu}{\partial\nn}\right)_\tau$ stands for the shear stress and $g:\R_+\to\R_+$ is a given slip bound function.
By a classical solution of \eqref{eq:stokes} we mean any couple of sufficiently smooth functions $(\uu,p)$ satisfying the differential equations and the boundary conditions in \eqref{eq:stokes}.

To give the weak formulation of \eqref{eq:stokes} we shall need the following function sets:
\begin{align}
\label{eq:def_space_VV}
\V(\Omega) &= \{\vv\in(H^1(\Omega))^2|~\vv=\ve 0 \mbox{ on }\Gamma,~v_\nu=0\mbox{ on }S\},\\
\label{eq:def_space_Vdiv}
\Vdiv(\Omega) &= \{\vv\in \V(\Omega)|~\div\vv=0 \mbox{ a.e. in }\Omega\},\\
L^2_0(\Omega) &= \{q\in L^2(\Omega)|~ \int_\Omega q = 0\},\\
L^2_+(S) &= \{\varphi\in L^2(S)|~ \varphi\ge 0 \mbox{ a.e. on }S\},\\
H^{1/2}(S) &= \{\varphi\in L^2(S)|~ \exists v\in H^1(\Omega),~v=0 \mbox{ on }\Gamma:~v=\varphi \mbox{ on }S\},\\
H^{1/2}_+(S) &= \{\varphi\in H^{1/2}(S)|~\varphi\ge 0 \mbox{ a.e. on }S\}.
\end{align}

\begin{remark}
If $\vv\in\V(\Omega)$ and $S\in\C^{1,1}$ then it is readily seen that $v_{\tau|S}\in H^{1/2}(S)$.
\end{remark}

\bigskip

{\it From now on we shall suppose that $S\in\C^{1,1}$.}

\bigskip

The trace space $H^{1/2}(S)$ is equipped with the norm
\[ \norm{\varphi}_{1/2,S} = \inf_{\substack{\vv\in \V(\Omega)\\v_\tau=\varphi}}|\vv|_{1,\Omega} = |\ww(\varphi)|_{1,\Omega}, \]
where $\ww(\varphi)\in \V(\Omega)$ is the solution to
\[ \left.\begin{aligned}
\Delta\ww(\varphi) &= \ve 0 \mbox{ in }\Omega,\\
\ww(\varphi) &= \ve 0 \mbox{ on }\Gamma,\\
w_\nu(\varphi) &= 0 \mbox{ on }S,\\
w_\tau(\varphi) &= \varphi \mbox{ on }S.
\end{aligned}\right\} \]

Further we introduce the following forms:
\begin{multline}
\label{eq:def_ab}
a(\uu,\vv) = \int_\Omega\nabla\uu:\nabla\vv, \quad b(\vv,q) = \int_\Omega q\div\vv, \quad j(\varphi,v_\tau) = \int_S g(\varphi)|v_\tau|,\\
\quad \uu,\vv\in (H^1(\Omega))^2,~q\in L^2(\Omega),~ \varphi\in H^{1/2}_+(S).
\end{multline}
We shall assume that $g:\R_+\to\R_+$ is \emph{continuous} and there exist positive constants $g_{min}<g_{max}$ such that
\begin{equation}
\label{eq:asm_g}
g_{min}\le g(\cdot) \le g_{max} \mbox{ in }\R_+.
\end{equation}

The weak formulation of \eqref{eq:stokes} reads as follows:
\begin{equation}
\tag{$\mathcal P$}
\label{eq:stokes_weak}
\left.\begin{array}{ll}
\multicolumn{2}{l}{\mbox{\it Find $(\uu,p)\in \V(\Omega)\times L^2_0(\Omega)$ such that}}\\\\
\forall\vv\in \V(\Omega):& a(\uu,\vv-\uu)-b(\vv-\uu,p)\hspace{30mm}\\\\
\multicolumn{2}{r}{+j(|u_\tau|,v_\tau)-j(|u_\tau|,u_\tau)\ge (\ff,\vv-\uu)_{0,\Omega},}\\\\
\forall q\in L^2_0(\Omega):& b(\uu,q)=0.
\end{array}\right\}
\end{equation}

We will show that under the above mentioned assumptions on $g$, \eqref{eq:stokes_weak} has at least one solution for any $\ff\in(L^2(\Omega))^2$.
To this end we use the weak variant of Schauder's fixed point theorem \cite{hhnl88}.

\medskip

For a given function $\varphi\in H^{1/2}_+(S)$ we consider the auxiliary problem:
\begin{equation}
\tag{$\pstavFix{\varphi}$}
\label{eq:stokes_weak_aux}
\left.\begin{array}{ll}
\multicolumn{2}{l}{\mbox{\it Find $(\uu^\varphi,p^\varphi)\in \V(\Omega)\times L^2_0(\Omega)$ such that}}\\\\
\forall\vv\in \V(\Omega):& a(\uu^\varphi,\vv-\uu^\varphi)-b(\vv-\uu^\varphi,p^\varphi)\hspace{30mm}\\\\
\multicolumn{2}{r}{+j(\varphi,v_\tau)-j(\varphi,u^\varphi_\tau)\ge (\ff,\vv-\uu^\varphi)_{0,\Omega},}\\\\
\forall q\in L^2_0(\Omega):& b(\uu^\varphi,q)=0.
\end{array}\right\}
\end{equation}

We know that for every $\varphi\in H^{1/2}_+(S)$ there exists a unique solution $(\uu^\varphi,p^\varphi)$ of \eqref{eq:stokes_weak_aux} and it satisfies (see \cite{fujita1994}):
\begin{equation}
\label{eq:estimate_aux}
\norm{\nabla\uu^\varphi}_{0,\Omega} + \norm{p^\varphi}_{0,\Omega} \le c(\norm{\ff}_{0,\Omega}+\norm{g(\varphi)}_{\infty,S}) \le \overline c,
\end{equation}
where $\overline c:=\overline c(\ff,g_{max})$ is a positive constant.

\bigskip

Let us define the mapping $\Psi:H^{1/2}_+(S)\to H^{1/2}_+(S)$ by
\[ \Psi(\varphi) = |u^\varphi_\tau| \mbox{ on }S. \]
Then \eqref{eq:stokes_weak} is equivalent to the problem of finding a fixed point of $\Psi$ in $H^{1/2}_+(S)$.
\\

\begin{theorem}
\label{th:prop_Psi}
The mapping $\Psi$ has the following properties:
\begin{itemize}
\item[\it (i)] $\Psi(B)\subset B$, where $B=\{\varphi\in H^{1/2}_+(S)|~\norm{\varphi}_{1/2,S} \le \overline c\}$ and $\overline c$ is the constant from \eqref{eq:estimate_aux}.
\item[\it (ii)] $\Psi$ is weakly continuous in $H^{1/2}_+(S)$, i.e.
\[ \varphi_k\weakly\varphi \mbox{ in }H^{1/2}(S),~\varphi_k,\varphi\in H^{1/2}_+(S) \Rightarrow \Psi(\varphi_k)\weakly\Psi(\varphi) \mbox{ in }H^{1/2}(S). \]
\end{itemize}
\end{theorem}

\begin{proof}
The property \textit{(i)} follows immediately from
\[ \norm{\,|u^\varphi_\tau|\,}_{1/2,S} \le \norm{u^\varphi_\tau}_{1/2,S} \le \norm{\nabla\uu^\varphi}_{0,\Omega} \le \overline c, \]
making use of \eqref{eq:estimate_aux}.

Let $(\uu^k,p^k)$ denote the solution to $(\pstavFix{\varphi_k})$, $\varphi_k\weakly\varphi$ in $H^{1/2}(S)$.
Since the sequence $\{(\uu^k,p^k)\}$ is bounded in $\V(\Omega)\times L^2_0(\Omega)$, there exists a subsequence (denoted by the index $k'$) such that
\[ \uu^{k'}\weakly \overline\uu \mbox{ in }(H^1(\Omega))^2,\quad p^{k'}\weakly \overline p \mbox{ in }L^2_0(\Omega),~k'\to\infty. \]
It is easy to show that $(\overline\uu,\overline p)$ is the solution of \eqref{eq:stokes_weak_aux}, i.e. $(\overline\uu,\overline p)=(\uu^\varphi,p^\varphi)$.
Indeed,
\begin{equation}
\label{eq:conv_bilin_forms}
\left.\begin{aligned}
\limsup_{k'\to\infty} a(\uu^{k'},\vv-\uu^{k'})&\le a(\overline\uu,\vv-\overline\uu),\\
(\ff,\vv-u^{k'})_{0,\Omega}&\to(\ff,\vv-\overline\uu)_{0,\Omega},\\
b(\vv-\uu^{k'},p^{k'})&\to b(\vv-\overline\uu,\overline p)\quad\forall\vv\in \V(\Omega).
\end{aligned}\right\}
\end{equation}
It remains to show that
\begin{equation}
\label{eq:conv_friction_term}
\int_S g(\varphi_{k'})(|v_\tau|-|u^{k'}_\tau|) \to \int_S g(\varphi)(|v_\tau|-|\overline u_\tau|),~k'\to\infty.
\end{equation}
To prove \eqref{eq:conv_friction_term}, we use that
\begin{equation}
\label{eq:strong_conv_varphi_g}
\varphi_{k'}\weakly\varphi \mbox{ in }H^{1/2}(S)\Rightarrow \varphi_{k''}\to\varphi \mbox{ a.e. on }S\Rightarrow g(\varphi_{k''})\to g(\varphi) \mbox{ a.e. on }S,
\end{equation}
where $\{\varphi_{k''}\}$ denotes a subsequence of $\{\varphi_{k'}\}$.
Moreover,
\begin{equation}
\label{eq:strong_conv_abs_u}
|u^{k'}_\tau|\to|\overline u_\tau| \mbox{ in }L^2(S).
\end{equation}
Clearly \eqref{eq:strong_conv_varphi_g} and \eqref{eq:strong_conv_abs_u} imply \eqref{eq:conv_friction_term}.
From \eqref{eq:conv_bilin_forms} and \eqref{eq:conv_friction_term} it follows that $(\overline\uu,\overline p)$ is a solution to \eqref{eq:stokes_weak_aux}. Since this solution is unique, then
\[ (\uu^k,p^k)\weakly(\overline\uu,\overline p) \mbox{ weakly in }(H^1(\Omega))^2\times L^2_0(\Omega), ~k\to\infty, \]
i.e. $(\overline\uu,\overline p)=(\uu^\varphi,p^\varphi)$.
Finally
\begin{multline*}
\uu^k\weakly\uu^\varphi \mbox{ in }(H^1(\Omega))^2 \Rightarrow |\uu^k|\weakly|\uu^\varphi| \mbox{ in }(H^1(\Omega))^2
\Rightarrow |u^k_\tau|\weakly|u^\varphi_\tau| \mbox{ in }H^{1/2}(S)\\
\Leftrightarrow \Psi(\varphi_k)\weakly\Psi(\varphi) \mbox{ in }H^{1/2}(S)
\end{multline*}
proving \textit{(ii)}.
\end{proof}

From the weak variant of Schauder's theorem and Theorem \ref{th:prop_Psi} it follows that there exists at least one fixed point of $\Psi$ in $H^{1/2}_+(S)$ and thus at least one solution to \eqref{eq:stokes_weak}.

\begin{remark}
Problem \eqref{eq:stokes_weak_aux} is well posed also for $\varphi\in L^2_+(S)$.
Similarly as in the previous theorem, one can show that $\Psi:L^2_+(S)\to L^2_+(S)$ is continuous:
\[ \varphi_k\to\varphi \mbox{ in }L^2(S),~\varphi_k,\varphi\in L^2_+(S)\Rightarrow\Psi(\varphi_k)\to\Psi(\varphi)\mbox{ in }L^2(S). \]
\end{remark}

Next we shall study under which conditions, problem \eqref{eq:stokes_weak} has a unique solution.

\begin{theorem}
\label{th:lip_Psi}
In addition to \eqref{eq:asm_g}, let $g:\R_+\to\R_+$ be Lipschitz continuous in $\R_+$:
\begin{equation}
\label{eq:lip_g}
\exists L>0: |g(x_1)-g(x_2)| \le L|x_1-x_2| ~\forall x_1,x_2\in\R_+.
\end{equation}
Then $\Psi:L^2_+(S)\to L^2_+(S)$ is Lipschitz continuous, as well:
\begin{equation}
\label{eq:lip_Psi}
\norm{\Psi(\varphi_1)-\Psi(\varphi_2)}_{0,S} \le c^2L\norm{\varphi_1-\varphi_2}_{0,S} ~\forall\varphi_1,\varphi_2\in L^2_+(S),
\end{equation}
where $L$ is from \eqref{eq:lip_g} and $c$ is the norm of the trace mapping $\tr:\V(\Omega)\to L^2(S)$, $\tr\vv=v_\tau$, assuming that $\V(\Omega)$ is equipped with the norm $|\cdot|_{1,\Omega}$.
\end{theorem}

\begin{proof}
Let $\varphi_1,\varphi_2\in L^2_+(S)$ and $(\uu^i,p^i)$ be solutions to $(\pstavFix{\varphi_i})$, $i=1,2$.
Then $\uu^i\in \Vdiv(\Omega)$ and
\[ a(\uu^i,\vv-\uu^i)+j(|u^i_\tau|,v_\tau) - j(|u^i_\tau|,u^i_\tau) \ge \scal{\ff}{\vv-\uu^i}_{0,\Omega}~\forall\vv\in \Vdiv(\Omega). \]
By a standard technique we obtain:
\begin{equation}
\label{eq:lip_est_for_Psi}
\begin{aligned}
|\uu^1-\uu^2|_{1,\Omega}^2 &= a(\uu^1-\uu^2,\uu^1-\uu^2) \le \int_S (g(\varphi_1)-g(\varphi_2))(|u^2_\tau|-|u^1_\tau|)\\
&\le_{\eqref{eq:lip_g}} L\norm{\varphi_1-\varphi_2}_{0,S}\norm{|u^2_\tau|-|u^1_\tau|}_{0,S}
\le L\norm{\varphi_1-\varphi_2}_{0,S}\norm{u^2_\tau-u^1_\tau}_{0,S}\\
&\le cL\norm{\varphi_1-\varphi_2}_{0,S}|\uu^1-\uu^2|_{1,\Omega}.
\end{aligned}
\end{equation}
From this and the fact that
\[ \norm{\Psi(\varphi_1)-\Psi(\varphi_2)}_{0,S} = \norm{|u^1_\tau|-|u^2_\tau|}_{0,S} \le c|\uu^1-\uu^2|_{1,\Omega} \]
we obtain \eqref{eq:lip_Psi}.
\end{proof}

\begin{corollary}
\label{th:corollary_iterations}
If $L<1/c^2$ then $\Psi$ is a contraction in $L^2_+(S)$.
Its unique fixed point belongs to the set $B_1=\{\varphi\in L^2_+(S)| ~\norm{\varphi}_{0,S}\le c\overline c\}$, since
\[ \norm{|u_\tau|}_{0,S} = \norm{u_\tau}_{0,S} \le c|\uu|_{1,\Omega} \le c\overline c, \]
as follows from \eqref{eq:estimate_aux}.
Consequently, the method of successive approximations
\[ \left.\begin{aligned} &\varphi_0\in L^2_+(S) \mbox{ given},\\ &\varphi_{k+1}:=\Psi(\varphi_k),~k=0,1,\ldots, \end{aligned}\right\} \]
converges in $L^2(S)$ to the unique fixed point $z$ of $\Psi$ in $L^2_+(S)$:
\begin{equation}
\label{eq:conv_succ_approx}
\varphi_k\to z \mbox{ in }L^2(S),\quad \Psi(z)=z \mbox{ on }S.
\end{equation}
In fact, $z\in H^{1/2}_+(S)$.
Indeed, let $\uu^k$, $\uu$ be the first component of the solution to $(\pstavFix{\varphi_k})$, and $(\pstavFix{z})$, respectively.
It is sufficient to show that $z=|u_\tau|\in H^{1/2}_+(S)$.
From \eqref{eq:lip_est_for_Psi} and \eqref{eq:conv_succ_approx} it follows:
\[ |\uu^k-\uu|_{1,\Omega} \le cL\norm{\varphi_k-z}_{0,S} \to 0,~k\to\infty, \]
so that
\[ \norm{\varphi_k-|u_\tau|}_{0,S} = \norm{|u^k_\tau|-|u_\tau|}_{0,S} \le \norm{u^k_\tau-u_\tau}_{0,S} \le c|\uu^k-\uu|_{1,\Omega} \to 0,~k\to\infty. \]
From this and \eqref{eq:conv_succ_approx} it follows that $z=|u_\tau|$ on $S$.
It is also readily seen that
\[ p^k\weakly p \mbox{ in }L^2(S) \]
and $(\uu,p)$ is the solution to $(\mathcal P)$.
\end{corollary}

\section{Four-field formulation of \eqref{eq:stokes_weak_aux} and \eqref{eq:stokes_weak}}
\label{sec:formulation_4field}

The pressure $p$ in the velocity-pressure formulation introduced in the previous section is the Lagrange multiplier associated with the incompressibility condition in $\Omega$.
This section presents another formulation involving two additional Lagrange multipliers $\sigma^\nu$, $\sigma^\tau$ defined on $S$ releasing the impermeability condition $u_\nu=0$ on $S$, and regularizing the non-differentiable functional $j$.
To this end we shall need the additional function spaces:
\begin{align}
W(\Omega) &= \{v\in H^1(\Omega)|~v=0\mbox{ on }\Gamma\},\\
\label{eq:def_space_WW}
\W(\Omega) &= W(\Omega)\times W(\Omega),\\
H^{-1/2}(S) &= (H^{1/2}(S))' \mbox{ (dual of }H^{1/2}(S)),\\
\HH^{1/2}(S) &= H^{1/2}(S)\times H^{1/2}(S),\\
\HH^{-1/2}(S) &= (\HH^{1/2}(S))'.
\end{align}
If $\ve\mu=(\mu_1,\mu_2)\in\HH^{-1/2}(S)$, $\pphi=(\varphi_1,\varphi_2)\in\HH^{1/2}(S)$ then
\[ \dual{\ve\mu}{\pphi}:= \dual{\mu_1}{\varphi_1} + \dual{\mu_2}{\varphi_2}. \]
Since $S\in\C^{1,1}$, the mapping
\[ \tr: \vv\mapsto (v_\nu,v_\tau), \mbox{ where }v_\nu=\vv_{|S}\cdot\nn,~v_\tau=\vv_{|S}\cdot\tt, \]
maps $\W(\Omega)$ onto $\HH^{1/2}(S)$.
If $\ve\mu=(\mu^\nu,\mu^\tau)\in\HH^{-1/2}(S)$ then
\[ \dual{\ve\mu}{\tr\vv}:=\dual{\mu^\nu}{v_\nu} + \dual{\mu^\tau}{v_\tau}. \]
Analogously to the previous section the space $\HH^{1/2}(S)$ is equipped with the norm 
\begin{equation}
\label{eq:norm_H12}
\norm{\pphi}_{1/2,S} = \inf_{\substack{\vv\in\W(\Omega)\\\tr\vv=\pphi}}|\vv|_{1,\Omega} = |\ww(\pphi)|_{1,\Omega},
\end{equation}
where $\ww(\pphi)$ in \eqref{eq:norm_H12} solves:
\begin{equation}
\label{eq:laplace_wphi}
\left.\begin{aligned}
\Delta\ww(\pphi)&=\ve 0 \mbox{ in }\Omega,\\
\ww(\pphi)&=\ve 0 \mbox{ on }\Gamma,\\
\tr\ww(\pphi)&=\pphi \mbox{ on }S.\qquad
\end{aligned}\right\}
\end{equation}
The standard dual norm in $\HH^{-1/2}(S)$ is given by
\[ \llbracket\ve\mu\rrbracket_{-1/2,S} = \sup_{\substack{\pphi\in\HH^{1/2}(S)\\\pphi\neq\ve 0}}\frac{\dual{\ve\mu}{\pphi}}{\norm{\pphi}_{1/2,S}} = \sup_{\substack{\vv\in\W(\Omega)\\\tr\vv\neq\ve 0}}\frac{\dual{\ve\mu}{\tr\vv}}{\norm{\tr\vv}_{1/2,S}}, \]
where $\norm{\ }_{1/2,S}$ is defined by \eqref{eq:norm_H12}.
One can introduce another norm on $\HH^{-1/2}(S)$, namely
\[ \norm{\ve\mu}_{-1/2,S} = \sup_{\substack{\vv\in\W(\Omega)\\\vv\neq\ve 0}}\frac{\dual{\ve\mu}{\tr\vv}}{|\vv|_{1,\Omega}}. \]
It is known \cite{girault-raviart} that 
\begin{equation}
\label{eq:equiv_dual_norms}
\llbracket\ve\mu\rrbracket_{-1/2,S} = \norm{\ve\mu}_{-1/2,S} \quad\forall\ve\mu\in\HH^{-1/2}(S).
\end{equation}

To regularize the functional $j(\varphi,\cdot)$ we introduce the bounded convex set $K(\varphi)$ defined by 
\[ K(\varphi)=\{\mu^\tau\in L^2(S)|~|\mu^\tau|\le g(\varphi) \mbox{ a.e. on }S\},\quad \varphi\in H^{1/2}_+(S). \]
It is readily seen that
\[ j(\varphi,v_\tau) = \int_S g(\varphi)|v_\tau| = \sup_{\mu^\tau\in K(\varphi)} \int_S \mu^\tau v_\tau. \]
Hence
\vspace{2mm}
\begin{equation}
\label{eq:ineq_j}
j(\varphi,v_\tau) \ge (\mu^\tau,v_\tau)_{0,S}\quad \forall\mu^\tau\in K(\varphi).
\end{equation}
\\
The \emph{four-field formulation} of \eqref{eq:stokes_weak_aux}, $\varphi\in H^{1/2}_+(S)$ reads as follows:
\begin{equation}
\tag{$\mathcal M^{\varphi}$}
\label{eq:stokes_weak_4field_psi}
\left.\begin{array}{ll}
\multicolumn{2}{l}{\mbox{\it Find $(\uu,p,\sigma^\nu,\sigma^\tau)\in \W(\Omega)\times L^2_0(\Omega)\times H^{-1/2}(S)\times K(\varphi)$}\mbox{ \it s.t.}}\\\\
\forall\vv\in \W(\Omega):& a(\uu,\vv)-b(\vv,p) - \dual{\sigma^\nu}{v_\nu} - \scal{\sigma^\tau}{v_\tau}_{0,S} = (\ff,\vv)_{0,\Omega},\\\\
\forall q\in L^2_0(\Omega):& b(\uu,q)=0,\\\\
\forall \mu^\nu\in H^{-1/2}(S):& \dual{\mu^\nu}{u_\nu} = 0,\\\\
\forall \mu^\tau\in K(\varphi):& \scal{\mu^\tau+\sigma^\tau}{u_\tau}_{0,S} \le 0.
\end{array}\right\}
\end{equation}
\\
Suppose that \eqref{eq:stokes_weak_4field_psi} has a solution.
In what follows we give its interpretation.
Clearly from $\eqref{eq:stokes_weak_4field_psi}_{2,3}$ we see that $\uu\in\Vdiv(\Omega)$, where $\Vdiv(\Omega)$ is defined by \eqref{eq:def_space_Vdiv}.
Using test functions $\vv\in\V(\Omega)$ in $\eqref{eq:stokes_weak_4field_psi}_1$ we get
\begin{equation}
\label{eq:momentum_sigma_tau}
a(\uu,\vv-\uu) - b(\vv-\uu,p) - (\sigma^\tau,v_\tau-u_\tau)_{0,S} = (\ff,\vv-\uu)_{0,\Omega} \quad \forall\vv\in\V(\Omega).
\end{equation}
From $\eqref{eq:stokes_weak_4field_psi}_4$ it follows that
\[ -(\sigma^\tau,u_\tau)_{0,S} = \sup_{\mu^\tau\in K(\varphi)} (\mu^\tau,u_\tau)_{0,S} = j(\varphi,u_\tau), \]
which together with \eqref{eq:ineq_j} yields
\[ -(\sigma^\tau,v_\tau-u_\tau)_{0,S} \le j(\varphi,v_\tau)-j(\varphi,u_\tau). \]
From this and \eqref{eq:momentum_sigma_tau} we obtain:
\[ a(\uu,\vv-\uu) - b(\vv-\uu,p) + j(\varphi,v_\tau)-j(\varphi,u_\tau) \ge (\ff,\vv-\uu)_{0,\Omega} \quad \forall\vv\in\V(\Omega), \]
i.e. the couple $(\uu,p)\in\V(\Omega)\times L^2_0(\Omega)$ solves \eqref{eq:stokes_weak_aux}.
The formal application of Green's formula to $\eqref{eq:stokes_weak_4field_psi}_1$ gives:
\[ \sigma^\nu = -p+\left(\frac{\partial\uu}{\partial\nn}\right)_\nu \quad\mbox{and}\quad \sigma^\tau = \left(\frac{\partial\uu}{\partial\nn}\right)_\tau \mbox{ on }S. \]
\begin{theorem}
\label{th:well_posed_4field_psi}
Problem \eqref{eq:stokes_weak_4field_psi} has a unique solution $(\uu,p,\sigma^\nu,\sigma^\tau)$ for any $\varphi\in H^{1/2}_+(S)$.
In addition, the couple $(\uu,p)$ solves \eqref{eq:stokes_weak_aux}.
\end{theorem}
\begin{proof}
To prove the existence and uniqueness of a solution to \eqref{eq:stokes_weak_4field_psi} it is sufficient to show that the bilinear form $c(\vv,q,\ve\mu):=-b(\vv,q)-\dual{\ve\mu}{\tr\vv}$ satisfies the LBB-condition.
It is well-known (see e.g. \cite{girault-raviart}) that
\[ \exists\gamma>0:~ \sup_{\substack{\vv\in(H^1_0(\Omega))^2\\\vv\neq\ve 0}} \frac{b(\vv,q)}{|\vv|_{1,\Omega}} \ge \gamma\norm{q}_{0,\Omega} ~\forall q\in L^2_0(\Omega) \]
and from \eqref{eq:equiv_dual_norms} we have
\[ \sup_{\substack{\vv\in\W(\Omega)\\\vv\neq\ve 0}} \frac{\dual{\ve\mu}{\tr\vv}}{|\vv|_{1,\Omega}} = \llbracket\ve\mu\rrbracket_{-1/2,S} ~\forall\ve\mu\in\HH^{-1/2}(S). \]
Then there exists a constant $\tilde\gamma>0$ such that
\[ \sup_{\substack{\vv\in\W(\Omega)\\\vv\neq\ve 0}} \frac{c(\vv,q,\ve\mu)}{|\vv|_{1,\Omega}} \ge \tilde\gamma(\norm{q}_{0,\Omega} + \llbracket\ve\mu\rrbracket_{-1/2,S}) ~\forall(q,\ve\mu)\in L^2_0(\Omega)\times\HH^{-1/2}(S) \]
as follows from Theorem 3.1 in \cite{howell-walkington}.
\end{proof}

\bigskip

Any quadruplet $(\uu,p,\sigma^\nu,\sigma^\tau)$ is said to be a \emph{solution of the problem with the solution dependent coefficient of friction} if it solves \eqref{eq:stokes_weak_4field_psi} with $\varphi=|u_\tau|$ on $S$:
\begin{equation}
\tag{$\mathcal M$}
\label{eq:stokes_weak_4field}
\left.\begin{array}{ll}
\multicolumn{2}{l}{(\uu,p,\sigma^\nu,\sigma^\tau)\in \W(\Omega)\times L^2_0(\Omega)\times H^{-1/2}(S)\times K(|u_\tau|)~\mbox{\it s.t.}}\\\\
\forall\vv\in \W(\Omega):& a(\uu,\vv)-b(\vv,p) - \dual{\sigma^\nu}{v_\nu} - \scal{\sigma^\tau}{v_\tau}_{0,S} = (\ff,\vv)_{0,\Omega},\\\\
\forall q\in L^2_0(\Omega):& b(\uu,q)=0,\\\\
\forall \mu^\nu\in H^{-1/2}(S):& \dual{\mu^\nu}{u_\nu} = 0,\\\\
\forall \mu^\tau\in K(|u_\tau|):& \scal{\mu^\tau+\sigma^\tau}{u_\tau}_{0,S} \le 0.
\end{array}\right\}
\end{equation}
On the basis of the results of Section \ref{sec:formulation_vp} and Theorem \ref{th:well_posed_4field_psi} we arrive at the following theorem.
\begin{theorem}
Problem \eqref{eq:stokes_weak_4field} has a solution.
In addition, if $g$ satisfies \eqref{eq:lip_g} with $L$ as in Corollary \ref{th:corollary_iterations}, then the solution is unique.
\end{theorem}

\section{Stability of solutions with respect to boundary variations}
\label{sec:shape_stability}

The aim of this section is to show that solutions to \eqref{eq:stokes_weak} and \eqref{eq:stokes_weak_4field} depend continuously on the shape of $\Omega$.
We shall suppose that only the part $S$ of $\partial\Omega$ where the slip conditions are prescribed, is subject to variations.
In addition, for the sake of simplicity of our presentation we shall assume that $S$ is represented by the graph of a function $\alpha$ which belongs to an appropriate class $\Uad$.
Here and in what follows $\Uad$ will be defined by
\begin{equation}
\label{eq:defUad}
\Uad = \{\alpha\in\C^{1,1}([0,1])|~0<\alpha_{min}\le\alpha\le\alpha_{max} \mbox{ in }[0,1],
|\alpha^{(j)}|\le C_j,~j=1,2 \mbox{ a.e. in }(0,1)\},
\end{equation}
where  $\alpha_{min}$, $\alpha_{max}$ and $C_j>0$, $j=1,2$ are given.
With any $\alpha\in\Uad$ we associate the domain
\[ \Omega(\alpha) = \{(x_1,x_2)\in\R^2|~ x_1\in(0,1),~x_2\in(\alpha(x_1),\omega)\}, \]
where $\omega>0$ is a constant which does not depend on $\alpha\in\Uad$.
The family of admissible domains consists of all $\Omega(\alpha)$ with $\alpha\in\Uad$.
We shall also assume that $\ff\in(L^2(\R^2))^2$.

\subsection{Stability of \eqref{eq:stokes_weak}}
\label{sec:stability_P}

Let $\alpha\in\Uad$ be given and denote by $(\uu(\alpha),p(\alpha))\in \V(\Omega(\alpha))\times L^2_0(\Omega(\alpha))$ a (not necessarily unique) solution to $(\pstav\alpha)$ defined in $\Omega(\alpha)$:
\begin{equation}
\label{eq:stav_alpha}
\tag{$\pstav\alpha$}
\left.
\begin{aligned}
&\forall\vv\in \V(\Omega(\alpha)):~&&a_\alpha(\uu(\alpha),\vv-\uu(\alpha))-b_\alpha(\vv-\uu(\alpha)),p(\alpha))\\
&&&\qquad  + j_\alpha(|u_\tau(\alpha)|,v_\tau) - j_\alpha(|u_\tau(\alpha)|,u_\tau(\alpha))\\
&&&\qquad \ge \scal{\ff}{\vv-\uu(\alpha)}_{0,\Omega(\alpha)},\\
&\forall q\in L^2_0(\Omega(\alpha)):~&&b_\alpha(\uu(\alpha),q)=0,
\end{aligned}
\right\}
\end{equation}
where $a_\alpha$, $b_\alpha$ and $j_\alpha$ are defined by \eqref{eq:def_ab} on $\Omega:=\Omega(\alpha)$, $S:=S(\alpha)$.

\begin{remark}
Let us note that the condition on $L$ ensuring the uniqueness of the solution to $(\pstav{\alpha})$ can be chosen to be independent of $\alpha\in\Uad$.
Indeed, the constant $c$ in \eqref{eq:lip_Psi} can be bounded from above uniformly with respect to $\alpha\in\Uad$ (see Lemma 2.19 in \cite{haslinger-makinen}).
\end{remark}

Let $\wOmega\supset\Omega(\alpha)$ $\forall\alpha\in\Uad$ be a hold-all domain and $\pi_\alpha\in\mathcal L(\V(\Omega(\alpha)),(H^1_0(\wOmega))^2)$ an extension mapping from $\Omega(\alpha)$ to $\wOmega$.
Since all $\Omega(\alpha)$, $\alpha\in\Uad$ satisfy the uniform cone property, there exists $\pi_\alpha$ whose norm can be estimated independently of $\alpha\in\Uad$.
Finally, the upper index ``${}^0$'' stands for the zero extension of functions from $\Omega(\alpha)$ to $\wOmega$.

\begin{theorem}
There exists a constant $c:=c(\ff,g_{max})>0$ independent of $\alpha\in\Uad$ such that
\begin{equation}
\label{eq:est_sol_alpha}
\norm{\pi_\alpha\uu(\alpha)}_{1,\wOmega} + \norm{p^0(\alpha)}_{0,\wOmega} \le c
\end{equation}
holds for any solution $(\uu(\alpha),p(\alpha))$ to $(\pstav\alpha)$.
\end{theorem}

\begin{proof}
The estimate of the first term in \eqref{eq:est_sol_alpha} follows from (3.4) in \cite{HaSt_given_slip}.
Similarly, the estimate of the second term in \eqref{eq:est_sol_alpha} is the same as (3.5) in \cite{HaSt_given_slip} using that $g\le g_{max}$ in $\R_+$.
\end{proof}

Let
\[ \G_{\mathcal P}:=\{(\alpha,\uu(\alpha),p(\alpha))|~\alpha\in\Uad,~(\uu(\alpha),p(\alpha)) \mbox{ solves }(\pstav\alpha)\} \]
be the graph of the generally multivalued solution mapping $\Phi:\alpha\mapsto(\uu(\alpha),p(\alpha))$, $\alpha\in\Uad$.

\begin{theorem}
\label{th:G_closed}
The graph $\G_{\mathcal P}$ is closed in the following sense:
\begin{equation*}
\left.\begin{aligned}
&\alpha_n\to\alpha \mbox{ in }\C^1([0,1]),~\alpha_n,\alpha\in\Uad,\\
&(\pi_{\alpha_n}\uu_n,p_n^0)\weakly(\ubar,\pbar) \mbox{ in }(H^1_0(\wOmega))^2\times L^2_0(\wOmega),\\
&\mbox{where }(\alpha_n,\uu_n,p_n):=(\alpha_n,\uu(\alpha_n),p(\alpha_n))\in\G_{\mathcal P}
\end{aligned}\right\} \Rightarrow (\ubar_{|\Omega(\alpha)},\pbar_{|\Omega(\alpha)}) \mbox{ solves }(\pstav\alpha)
\end{equation*}
and hence $(\alpha,\ubar_{|\Omega(\alpha)},\pbar_{|\Omega(\alpha)})\in\G_{\mathcal P}$.
\end{theorem}

\begin{proof}
Let $\vv\in \V(\Omega(\alpha))$ be given.
From Lemma 3 in \cite{HaSt_given_slip} we know that there exist: a sequence $\{\vv_k\}$, $\vv_k\in(H^1(\wOmega))^2$ and a filter of indices $\{n_k\}$ such that
\begin{equation}
\label{eq:conv_vk}
\vv_k\to\vv \mbox{ in }(H^1(\wOmega))^2,~k\to\infty
\end{equation}
and 
\[ \vv_{k|\Omega(\alpha_{n_k})}\in \V(\Omega(\alpha_{n_k})). \]
Therefore $\vv_{k|\Omega(\alpha_{n_k})}$ can be used as a test function in $(\pstav{\alpha_{n_k}})$:
\begin{equation}
\label{eq:Pnk}
\left.
\begin{aligned}
&a_{\alpha_{n_k}}(\uu_{n_k},\vv_k-\uu_{n_k})-b_{\alpha_{n_k}}(\vv_k-\uu_{n_k},p_{n_k})\\
&\qquad  + j_{\alpha_{n_k}}(|u_{{n_k}\tau}|,v_{k\tau}) - j_{\alpha_{n_k}}(|u_{n_k\tau}|,u_{n_k\tau})\\
&\qquad  \ge \scal{\ff}{\vv_k-\uu_{n_k}}_{0,\Omega(\alpha_{n_k})},\\\\
&\forall q\in L^2_0(\Omega(\alpha_{n_k})):~b_{\alpha_{n_k}}(\uu_{n_k},q)=0.
\end{aligned}
\right\}
\end{equation}
Denote $(\uu(\alpha),p(\alpha)):=(\ubar_{|\Omega(\alpha)},\pbar_{|\Omega(\alpha)})$.
The fact that $\uu(\alpha)\in\Vdiv(\Omega(\alpha))$ and the following limit passages can be proven exactly as in \cite{HaSt_given_slip}:
\[ \limsup_{k\to\infty} a_{\alpha_{n_k}}(\uu_{n_k},\vv_k-\uu_{n_k}) \le a_\alpha(\uu(\alpha),\vv-\uu(\alpha)), \]
\[ \lim_{k\to\infty} b_{\alpha_{n_k}}(\vv_k-\uu_{n_k},p_{n_k}) = b_\alpha(\vv-\uu(\alpha),p(\alpha)), \]
\[ \lim_{k\to\infty} \scal{\ff}{\vv_k-\uu_{n_k}}_{0,\Omega(\alpha_{n_k})} = \scal{\ff}{\vv-\uu(\alpha)}_{0,\Omega(\alpha_{n_k})}. \]
It remains to pass to the limit in the slip terms given by $j$.

From \eqref{eq:conv_vk} and weak convergence $\pi_{\alpha_{n_k}}(\uu_{n_k})\weakly\ubar$ in $(H^1(\wOmega))^2$ it follows (see Lemma 2.21, \cite{haslinger-makinen}):
\begin{equation*}
\left.\begin{aligned}
\uu_{n_k|S(\alpha_{n_k})}\circ\alpha_{n_k} &\to \ubar_{|S(\alpha)}\circ\alpha\\
\vv_{k|S(\alpha_{n_k})}\circ\alpha_{n_k} &\to \vv_{|S(\alpha)}\circ\alpha
\end{aligned}\right\} \mbox{ in }(L^2(0,1))^2
\end{equation*}
and also
\begin{equation}
\label{eq:conv_unktau_vktau}
\left.\begin{aligned}
u_{n_k\tau}\circ\alpha_{n_k} := (\uu_{n_k|S(\alpha_{n_k})}\cdot\ve\tau^{\alpha_{n_k}})\circ\alpha_{n_k} &\to (\ubar_{|S(\alpha)}\cdot\ve\tau^\alpha)\circ\alpha=\overline u_\tau\circ\alpha\\
v_{k\tau}\circ\alpha_{n_k}:=(\vv_{k|S(\alpha_{n_k})}\cdot\ve\tau^{\alpha_{n_k}})\circ\alpha_{n_k} &\to (\vv_{|S(\alpha)}\cdot\ve\tau^\alpha)\circ\alpha=v_\tau\circ\alpha
\end{aligned}\right\} \mbox{ in }L^2(0,1),
\end{equation}
using that $\ve\tau^{\alpha_{n_k}}\circ\alpha_{n_k}\rightrightarrows\ve\tau^\alpha\circ\alpha$ (uniformly) in $[0,1]$, where $\ve\tau^\beta$ stands for the unit tangential vector to $S(\beta)$, $\beta\in\Uad$ (see \cite{HaSt_given_slip}).
Consequently, there exists a subsequence of $\{u_{n_k\tau}\circ\alpha_{n_k}\}$ (denoted by the same symbol) such that
\begin{equation}
\label{eq:conv_unktau_ae}
u_{n_k\tau}\circ\alpha_{n_k}\to\overline u_\tau\circ\alpha \mbox{ a.e. in }(0,1).
\end{equation}
Hence
\begin{multline*}
j_{\alpha_{n_k}}(|u_{n_k\tau}|,v_{k\tau})=\int_0^1 g(|u_{n_k\tau}\circ\alpha_{n_k}|)v_{k\tau}\circ\alpha_{n_k}\sqrt{1+(\alpha_{n_k}')^2} dx_1\\
\to \int_0^1 g(|\overline u_\tau\circ\alpha|)v_\tau\circ\alpha\sqrt{1+(\alpha')^2}dx_1 = j_\alpha(|\overline u_\tau|,v_\tau),~k\to\infty,
\end{multline*}
making use of $\eqref{eq:conv_unktau_vktau}_2$, \eqref{eq:conv_unktau_ae} and convergence $\alpha_n\to\alpha$ in $\C^1([0,1])$.
The limit passage for the second slip term in \eqref{eq:Pnk} can be done in the same way.
\end{proof}

\subsection{Stability of \eqref{eq:stokes_weak_4field}}

Let $\Uad$ be defined again by \eqref{eq:defUad}.
We keep notation of Subsection \ref{sec:stability_P}, i.e. $\Omega(\alpha)$, $S(\alpha)$, $j_\alpha$, $a_\alpha$, $b_\alpha$, $\nn^\alpha$, $\tt^\alpha$ have the same meaning.
In addition, $\dual{\ }{\ }_\alpha$ will denote the duality pairing between $H^{-1/2}(S(\alpha))$ and $H^{1/2}(S(\alpha))$.
If $\mu^\nu\in H^{-1/2}(S(\alpha))$ and $\vv\in(H^1(\Omega(\alpha)))^2$ then
\[ \dual{\mu^\nu}{v_\nu}_\alpha := \dual{\mu^\nu}{\vv_{|S(\alpha)}\cdot\nn^\alpha}_\alpha \quad \forall\alpha\in\Uad \]
and similarly for $\dual{\mu^\tau}{v_\tau}_\alpha$, $\mu^\tau\in H^{-1/2}(S(\alpha))$.
Problem \eqref{eq:stokes_weak_4field} formulated on $\Omega:=\Omega(\alpha)$ will be denoted by $(\mathcal M(\alpha))$.
\bigskip

Let
\begin{equation*}
\G_{\mathcal M} = \{(\alpha,\uu(\alpha),p(\alpha),\sigma^\nu(\alpha),\sigma^\tau(\alpha))|~\alpha\in\Uad,~(\uu(\alpha),p(\alpha),\sigma^\nu(\alpha),\sigma^\tau(\alpha))\mbox{ solves }(\mathcal M(\alpha))\}
\end{equation*}
be the graph of the respective solution mapping.

\begin{theorem}
\label{th:domain_dependence_4field}
The graph $\G_{\mathcal M}$ is closed in the following sense:
Let
\begin{equation}
\label{eq:conv_alpha}
\alpha_n\to\alpha \mbox{ in }\C^1([0,1]), ~\alpha_n,\alpha\in\Uad
\end{equation}
and 
\begin{equation}
\label{eq:conv_up_alpha}
(\pi_{\alpha_n}\uu_n,p_n^0) \weakly (\ubar,\pbar) \mbox{ in }(H^1_0(\wOmega))^2\times L^2_0(\wOmega).
\end{equation}
Then also
\begin{align}
\label{eq:conv_sigma_nu}
\dual{\sigma^\nu_n}{v_\nu}_{\alpha_n} &\to \dual{\sigma^\nu(\alpha)}{v_\nu}_\alpha,\\
\label{eq:conv_sigma_tau}
\scal{\sigma^\tau_n}{v_\tau}_{0,S(\alpha_n)} &\to \scal{\sigma^\tau(\alpha)}{v_\tau}_{0,S(\alpha)}
\end{align}
holds for every $\vv\in(H^1_0(\wOmega))^2$, where $(\uu_n,p_n,\sigma^\nu_n,\sigma^\tau_n)$ is a solution to $(\mathcal M(\alpha_n))$, $n=1,\ldots$
In addition, the quadruplet $(\ubar_{|\Omega(\alpha)},\pbar_{|\Omega(\alpha)},\sigma^\nu(\alpha),\sigma^\tau(\alpha))$ solves $(\mathcal M(\alpha))$.
\end{theorem}

\bigskip

To prove this theorem we shall need the following auxiliary result.
\begin{lemma}
\label{th:lemma_conv_linfty}
Let $g,h\in L^\infty((0,1))$, $g\ge 0$ be such that $|h|\le g$ a.e. in $(0,1)$ and $\{g_n\}$, $g_n\in L^\infty((0,1))$ be a sequence of nonnegative functions, $g_n\to g$ a.e. in $(0,1)$. 
Then there exists a sequence $\{h_n\}$, $|h_n|\le g_n$ a.e. in $(0,1)$ such that $h_n\to h$ a.e. in $(0,1)$.
\end{lemma}
\begin{proof}
The interval $(0,1)$ will be decomposed as follows: $(0,1)=\cup_{i=1}^5 I_i$, where
\[
\begin{aligned}
I_1 &= \{x\in(0,1)|~h(x)=0\},\\
I_2 &= \{x\in(0,1)|~h(x)=g(x),~g(x)\neq 0\},\\
I_3 &= \{x\in(0,1)|~h(x)=-g(x),~g(x)\neq 0\},\\
I_4 &= \{x\in(0,1)|~h(x)\in(0,g(x))\},\\
I_5 &= \{x\in(0,1)|~h(x)\in(-g(x),0)\}.
\end{aligned}
\]
The sequence $\{h_n\}$ defined by
\[
h_n = \begin{cases}
0 &\mbox{ on }I_1,\\
g_n&\mbox{ on }I_2,\\
-g_n&\mbox{ on }I_3,\\
\min(h,g_n)&\mbox{ on }I_4,\\
\max(h,g_n)&\mbox{ on }I_5
\end{cases}
\]
has the required properties.
\end{proof}

\bigskip

\begin{proof}[Proof of Theorem \ref{th:domain_dependence_4field}]
In virtue of Theorem \ref{th:G_closed} we have to prove \eqref{eq:conv_sigma_nu} and \eqref{eq:conv_sigma_tau} only.
We use the formulation of $(\mathcal M(\alpha_n))$:
\begin{equation}
\label{eq:stokes_weak_4field_n}
\left.\begin{array}{ll}
\multicolumn{2}{l}{(\uu_n,p_n,\sigma^\nu_n,\sigma^\tau_n)\in \W(\Omega(\alpha_n))\times L^2_0(\Omega(\alpha_n))\times H^{-1/2}(S(\alpha_n))\times K(|u_{n\tau|S(\alpha_n)}|):}\\\\
\forall\vv\in (H^1_0(\wOmega))^2:& \dual{\sigma^\nu_n}{v_\nu}_{\alpha_n} =  a_{\alpha_n}(\uu_n,\vv) - b_{\alpha_n}(\vv,p_n) \\
&\hspace{\stretch{1}} - \scal{\sigma^\tau_n}{v_\tau}_{0,S(\alpha_n)} - (\ff,\vv)_{0,\Omega(\alpha_n)},\\\\
\forall q\in L^2_0(\Omega(\alpha_n)):& b_{\alpha_n}(\uu_n,q)=0,\\\\
\forall \mu^\nu\in H^{-1/2}(S(\alpha_n)):& \dual{\mu^\nu}{u_{n\nu}}_{\alpha_n} = 0,\\\\
\forall \mu^\tau\in K(|u_{n\tau|S(\alpha_n)}|):& \scal{\mu^\tau+\sigma^\tau_n}{u_{n\tau}}_{0,S(\alpha_n)} \le 0,
\end{array}\right\}
\end{equation}
where
\[ K(|u_{n\tau|S(\alpha_n)}|) = \{\mu^\tau\in L^2(S(\alpha_n))|~|\mu^\tau|\le g(|\uu_{n|S(\alpha_n)}\cdot\tt^{\alpha_n}|)\mbox{ a.e. in }S(\alpha_n)\}. \]
Observe that one can use test functions $\vv\in (H^1_0(\wOmega))^2$ in $\eqref{eq:stokes_weak_4field_n}_1$ since $\W(\Omega(\beta))=(H^1_0(\wOmega))^2_{|\Omega(\beta)}$ $\forall\beta\in\Uad$.
\bigskip

Denote $\uu(\alpha):=\ubar_{|\Omega(\alpha)}$ and $p(\alpha):=\pbar_{|\Omega(\alpha)}$.
Letting $n\to\infty$ in $\eqref{eq:stokes_weak_4field_n}_1$ we obtain:
\begin{equation}
\lim_{n\to\infty}\dual{\sigma^\nu_n}{v_\nu}_{\alpha_n} =  a_\alpha(\uu(\alpha),\vv) - b_\alpha(\vv,p(\alpha)) - \scal{\ff}{\vv}_{0,\Omega(\alpha)} - \lim_{n\to\infty}\scal{\sigma^\tau_n}{v_\tau}_{0,S(\alpha_n)}.
\end{equation}
Further
\begin{multline}
\lim_{n\to\infty}\scal{\sigma^\tau_n}{v_\tau}_{0,S(\alpha_n)} = \lim_{n\to\infty}\int_{S(\alpha_n)}\sigma^\tau_n\vv\cdot\tt^{\alpha_n}\\
= \lim_{n\to\infty}\int_0^1\hat\sigma^\tau_n(\vv_{|S(\alpha_n)}\cdot\tt^{\alpha_n})\circ\alpha_n\sqrt{1+(\alpha_n')^2}~dx_1,
\end{multline}
where $\hat\sigma^\tau_n=\sigma^\tau_n\circ\alpha_n$.
Clearly $|\hat\sigma^\tau_n|\le g(|\widehat{u_{n\tau}}_{|S(\alpha_n)}|)$ a.e. in $(0,1)$ with $\widehat{u_{n\tau}}_{|S(\alpha_n)}=(\uu_{n|S(\alpha_n)}\cdot\tt^{\alpha_n})\circ\alpha_n$.
Hence there exists a subsequence of $\{\hat\sigma^\tau_n\}$ (denoted by the same symbol) and an element $\hat\theta^\tau\in L^2((0,1))$ such that
\begin{equation}
\label{eq:conv_sigma_tau_hat}
\hat\sigma^\tau_n\weakly \hat\theta^\tau \mbox{ in }L^2((0,1)).
\end{equation}
It is easy to show that $|\hat\theta^\tau|\le g(\widehat{u_\tau(\alpha)}_{|S(\alpha)})$ a.e. in $(0,1)$ making use of \eqref{eq:conv_unktau_ae} and \eqref{eq:conv_sigma_tau_hat}.
Since $\alpha_n\to\alpha$ in $\C^1([0,1])$ and $\tt^{\alpha_n}\circ\alpha_n\rightrightarrows\tt^\alpha\circ\alpha$ in $[0,1]$, we finally obtain:
\begin{equation}
\label{eq:lim_sigma_tau_n}
\lim_{n\to\infty}\scal{\sigma^\tau_n}{v_\tau}_{0,S(\alpha_n)} = \int_0^1\hat\theta^\tau(\vv_{|S(\alpha)}\cdot\tt^\alpha)\circ\alpha\sqrt{1+(\alpha')^2}~dx_1 = \scal{\theta^\tau}{v_\tau}_{0,S(\alpha)},
\end{equation}
where $\theta^\tau$ is defined by $\hat\theta^\tau=\theta^\tau\circ\alpha$.
Thus $\theta^\tau\in K(|u_\tau(\alpha)_{|S(\alpha)}|)$.

Since $\uu(\alpha)\in\Vdiv(\Omega(\alpha))$ we have
\begin{equation}
\label{eq:vanish_b_alpha}
b_\alpha(\uu(\alpha),q)  = 0 ~\forall q\in L^2_0(\Omega(\alpha)),
\end{equation}
and
\begin{equation}
\label{eq:vanish_u_nu_alpha}
\dual{\mu^\nu}{u_\nu(\alpha)}_{\alpha} = 0 ~\forall\mu^\nu\in H^{-1/2}(S(\alpha)).
\end{equation}
\\
To verify the last inequality $\eqref{eq:stokes_weak_4field_n}_4$, let $\mu^\tau\in K(|u_\tau(\alpha)_{|S(\alpha)}|)$ be arbitrary.
Then $|\hat\mu^\tau|\le g(|\widehat{u_\tau(\alpha)}_{|S(\alpha)}|)$ a.e. in $(0,1)$.
From Lemma \ref{th:lemma_conv_linfty} and the Lebesgue theorem we know that there exists a sequence $\{\hat\mu^\tau_n\}$, $|\hat\mu^\tau_n|\le g(|\widehat{u_{n\tau}}_{|S(\alpha_n)}|)$ a.e. in $(0,1)$ such that
\[ \hat\mu^\tau_n \to \hat\mu^\tau \mbox{ in }L^2((0,1)). \]
Consequently,
\begin{multline*}
0\ge \scal{\mu^\tau_n+\sigma^\tau_n}{u_{n\tau}}_{0,S(\alpha_n)} = \int_0^1(\hat\mu^\tau_n+\hat\sigma^\tau_n)(\uu_{n|S(\alpha_n)}\cdot\tt^{\alpha_n})\circ\alpha_n\sqrt{1+(\alpha_n')^2}~dx_1\\
\xrightarrow{n\to\infty} \int_0^1(\hat\mu^\tau+\hat\theta^\tau)(\uu(\alpha)_{|S(\alpha)}\cdot\tt^\alpha)\circ\alpha\sqrt{1+(\alpha')^2}~dx_1
= \scal{\mu^\tau+\theta^\tau}{u_\tau(\alpha)}_{0,S(\alpha)}.
\end{multline*}
So far we have shown that there exists a function $\theta^\tau\in K(|u_\tau(\alpha)_{|S(\alpha)}|)$ such that
\begin{equation}
\label{eq:conv_alpha_so_far}
\left.\begin{aligned}
&\lim_{n\to\infty}\dual{\sigma^\nu_n}{v_\nu}_{\alpha_n} =  a_\alpha(\uu(\alpha),\vv) - b_\alpha(\vv,p(\alpha))\\
&\hspace{3cm} - \scal{\theta^\tau}{v_\tau}_{0,S(\alpha)} - \scal{\ff}{\vv}_{0,\Omega(\alpha)} ~~\forall\vv\in(H^1_0(\wOmega))^2,\\\\
&\scal{\mu^\tau+\theta^\tau}{u_\tau(\alpha)}_{0,S(\alpha)} \le 0 ~~\forall\mu^\tau\in K(|u_\tau(\alpha)_{|S(\alpha)}|),\\\\
&+ \eqref{eq:vanish_b_alpha}\mbox{ and }\eqref{eq:vanish_u_nu_alpha}.
\end{aligned}\right\}
\end{equation}
To finish the proof we have to show that $\theta^\tau=\sigma^\tau(\alpha)$.
Indeed, then \eqref{eq:conv_sigma_nu} follows from the fact that the right hand side of $\eqref{eq:conv_alpha_so_far}_1$ is equal to $\dual{\sigma^\nu(\alpha)}{v_\nu}_{\alpha}$ and \eqref{eq:conv_sigma_tau} from \eqref{eq:lim_sigma_tau_n}.
We use again Lemma 3 from \cite{HaSt_given_slip}: for any $\vv\in\V(\Omega(\alpha))$ there exists a sequence $\{\vv_k\}$, $\vv_k\in(H^1_0(\wOmega))^2$ and a filter of indices $\{n_k\}$ such that
\begin{equation}
\vv_k \to \vbar \mbox{ in }(H^1_0(\wOmega))^2
\end{equation}
and
\begin{equation}
\vv_{k|\Omega(\alpha_{n_k})}\in\V(\Omega(\alpha_{n_k})),
\end{equation}
where $\vbar_{|\Omega(\alpha)}=\vv$.
The definition of $(\mathcal M(\alpha_{n_k}))$ yields:
\begin{equation*}
0=\dual{\sigma^\nu_{n_k}}{v_{k\nu}}_{\alpha_{n_k}} =  a_{\alpha_{n_k}}(\uu_{n_k},\vv_k)
- b_{\alpha_{n_k}}(\vv_k,p_{n_k}) - \scal{\sigma^\tau_{n_k}}{v_{k\tau}}_{0,S(\alpha_{n_k})}  - \scal{\ff}{\vv_k}_{0,\Omega(\alpha_{n_k})}.
\end{equation*}
Letting $k\to\infty$ and using \eqref{eq:conv_alpha_so_far} we obtain as before:
\begin{equation*}
0 = a_\alpha(\uu(\alpha),\vv) - b_\alpha(\vv,p(\alpha)) - \scal{\theta^\tau}{v_\tau}_{0,S(\alpha)} - \scal{\ff}{\vv}_{0,\Omega(\alpha)}
= \scal{\sigma^\tau(\alpha)}{v_\tau}_{0,S(\alpha)} - \scal{\theta^\tau}{v_\tau}_{0,S(\alpha)}
\end{equation*}
holds for any $\vv\in\V(\Omega(\alpha))$.
Hence $\sigma^\tau(\alpha)=\theta^\tau$.
\end{proof}

\section{Application of the stability property in optimal shape design problems}
\label{sec:shape_optimization}

On the basis of the results of Section \ref{sec:shape_stability} it is easy to prove the existence of solutions to a class of optimal shape design problems for systems governed by the Stokes equation with a solution dependent slip bound.

Let $\J_{\mathcal P}:\G_{\mathcal P}\to\R$ and $\J_{\mathcal M}:\G_{\mathcal M}\to\R$ be cost functionals defined on the graphs of the solution mappings corresponding to $(\mathcal P(\alpha))$ and $(\mathcal M(\alpha))$, $\alpha\in\Uad$, respectively.
Shape optimization problems with $(\mathcal P(\alpha))$, $(\mathcal M(\alpha))$ as the state relation read as follows:
\begin{equation}
\tag{$\mathbb P_{\mathcal P}$}
\label{eq:PoptP}
\left.
\begin{aligned}
&\mbox{\it Find }\ve z^*\in\G_{\mathcal P}\mbox{ \it such that}\\\\
&\J_{\mathcal P}(\ve z^*) \le \J_{\mathcal P}(\ve z)~\forall \ve z\in\G_{\mathcal P},
\end{aligned}
\right\}
\end{equation}
where $\ve z^*=(\alpha^*,\uu(\alpha^*),p(\alpha^*))$, $\ve z=(\alpha,\uu(\alpha),p(\alpha))$, and
\begin{equation}
\tag{$\mathbb P_\mathcal M$}
\label{eq:PoptM}
\left.
\begin{aligned}
&\mbox{\it Find }\ve z^*\in\G_{\mathcal M}\mbox{ \it such that}\\\\
&\J_{\mathcal M}(\ve z^*) \le \J_{\mathcal M}(\ve z)~\forall \ve z\in\G_{\mathcal M},
\end{aligned}
\right\}
\end{equation}
where $\ve z^*=(\alpha^*,\uu(\alpha^*),p(\alpha^*),\sigma^\nu(\alpha^*),\sigma^\tau(\alpha^*))$, $\ve z=(\alpha,\uu(\alpha),p(\alpha),\sigma^\nu(\alpha),\sigma^\tau(\alpha))$, respectively.

To prove the existence of solutions to \eqref{eq:PoptP} and \eqref{eq:PoptM} we use compactness and lower semicontinuity arguments.

Convergence in $\G_{\mathcal P}$ and $\G_{\mathcal M}$ will be introduced using the results of Theorem \ref{th:G_closed} and \ref{th:domain_dependence_4field}.
We say that $\ve z_n\to\ve z$, $\ve z_n$, $\ve z\in\G_{\mathcal P}$ if \eqref{eq:conv_alpha} and \eqref{eq:conv_up_alpha} hold.
Analogously, $\ve z_n\to\ve z$, $\ve z_n$, $\ve z\in\G_{\mathcal M}$ if \eqref{eq:conv_alpha}--\eqref{eq:conv_sigma_tau} hold.

From the definition of $\Uad$ and the Arzel\`a-Ascoli theorem we see that $\Uad$ is a compact subset of $\C^1$.
This, together with \eqref{eq:est_sol_alpha} and Theorem \ref{th:G_closed} and \ref{th:domain_dependence_4field} proves the following result.
\begin{theorem}
The sets $\G_{\mathcal P}$ and $\G_{\mathcal M}$ are compact with respect to convergences introduced above.
\end{theorem}

We say that $\J_{\mathcal P}$ and $\J_{\mathcal M}$ are lower semicontinuous functionals on $\G_{\mathcal P}$, and $\G_{\mathcal M}$, respectively if
\begin{equation}
\label{eq:JP_lsc}
\ve z_n\to\ve z, ~\ve z_n,~\ve z\in\G_{\mathcal P} \Rightarrow \liminf_{n\to\infty} \J_{\mathcal P}(\ve z_n)\ge\J_{\mathcal P}(\ve z),
\end{equation}
\begin{equation}
\label{eq:JM_lsc}
\ve z_n\to\ve z, ~\ve z_n,~\ve z\in\G_{\mathcal M} \Rightarrow \liminf_{n\to\infty} \J_{\mathcal M}(\ve z_n)\ge\J_{\mathcal M}(\ve z).
\end{equation}
\begin{theorem}
Let the functionals $\J_{\mathcal P}$, $\J_{\mathcal M}$ satisfy \eqref{eq:JP_lsc}, and \eqref{eq:JM_lsc}, respectively.
Then there exists a solution to \eqref{eq:PoptP} and \eqref{eq:PoptM}.
\end{theorem}
Proof is straightforward.

\section*{Conclusions}

In the first part of this paper we analyzed the mathematical model of  the Stokes system with a threshold slip boundary condition whose slip bound depends on the solution itself.
We used two weak formulations of this problem: the standard velocity-pressure formulation and the extended formulation in terms of the velocity, pressure, shear and normal stress.
We proved the existence of a solution for a large class of functions representing the slip bound and studied under which conditions the  solution is unique.
In the second part of the paper we analyzed how  solutions to both weak formulations depend on the geometry of the problem, in particular on the shape  of the slip part $S$ of the boundary.
Using an appropriate parametrization of $S$  we proved that the graphs of the respective  solution mappings are compact in an appropriate topology.
This result plays the key role in shape optimization.

\paragraph{Acknowledgement}
The work of J.H. was supported by the Czech Science Foundation grant No. P201/12/0671.
J.S.~acknowledges the institutional support RVO:46747885.

%
\bibliographystyle{zamm-title}
\bibliography{ref}

\providecommand{\WileyBibTextsc}{}
\let\textsc\WileyBibTextsc
\providecommand{\othercit}{}
\providecommand{\jr}[1]{#1}
\providecommand{\etal}{~et~al.}


\begin{thebibliography}{[10]}

\othercit
\bibitem{BM13}
 \textsc{M.~Bul{\'\i}{\v c}ek} and  \textsc{J.~M{\'a}lek},
On unsteady internal flows of {B}ingham fluids subject to threshold slip on the
  impermeable boundary,
 in: Advances in Mathematical Fluid Mechanics, , Recent Developments of
  Mathematical Fluid Mechanics (Springer, ),
accepted.


\bibitem{fujita1994}
 \textsc{H.~Fujita},
A mathematical analysis of motions of viscous incompressible fluid under leak
  or slip boundary conditions,
 \jr{RIMS K\^oky\^uroku} \textbf{888}, 199--216 (1994).


\bibitem{fujita2002}
 \textsc{H.~Fujita},
A coherent analysis of {S}tokes flows under boundary conditions of friction
  type,
 \jr{Journal of Computational and Applied Mathematics} \textbf{149}, 57--69
  (2002).


\othercit
\bibitem{girault-raviart}
 \textsc{V.~Girault} and  \textsc{P.\,A. Raviart},
Finite {E}lement {A}pproximation of the {N}avier-{S}tokes {E}quations, Lecture
  Notes in Mathematics,  Vol.\,749 (Springer, 1979).


\othercit
\bibitem{haslinger-makinen}
 \textsc{J.~Haslinger} and  \textsc{R.~M{\"a}kinen},
Introduction to Shape Optimization: Theory, Approximation, and Computation,
  Advances in Design and Control, DC07 (Society for Industrial Mathematics,
  2003).


\bibitem{HaSt_given_slip}
 \textsc{J.~Haslinger},  \textsc{J.~Stebel},  and  \textsc{T.~Sassi},
Shape optimization for {S}tokes problem with threshold slip,
 \jr{Applications of Mathematics} \textbf{59}(6), 631--652 (2014).


\bibitem{HL03}
 \textsc{H.~Hervet} and  \textsc{L.~Leger},
Flow with slip at the wall: from simple to complex fluids,
 \jr{C.R. Physique} \textbf{4}, 241--249 (2003).


\othercit
\bibitem{hhnl88}
 \textsc{I.~Hlav{\'a}{\v c}ek},  \textsc{J.~Haslinger},  \textsc{J.~Ne{\v
  c}as},  and  \textsc{J.~Lov{\'\i}{\v s}ek},
Solution of variational inequalities in mechanics, Applied Mathematical
  Sciences,  Vol.\,66 (Springer, 1988).


\bibitem{howell-walkington}
 \textsc{J.\,S. Howell} and  \textsc{N.\,J. Walkington},
Inf--sup conditions for twofold saddle point problems,
 \jr{Numerische Mathematik} \textbf{118}(4), 663--693 (2011).


\bibitem{RR99}
 \textsc{I.~Rao} and  \textsc{K.~Rajagopal},
The effect of the slip boundary conditions on the flow of fluids in a channel,
 \jr{Acta Mechanica} \textbf{135}, 113--126 (1999).


\bibitem{saito}
 \textsc{N.~Saito},
On the {S}tokes equation with the leak and slip boundary conditions of friction
  type: regularity of solutions,
 \jr{Publ. Res. Inst. Math. Sci.} \textbf{40}(2), 345--383 (2004).


\end{thebibliography}
 
%
%

\end{document}